\documentclass[12pt]{amsart}

\usepackage{amsmath}
\usepackage{amssymb}
\usepackage[curve]{xypic}
\usepackage{caption}
\usepackage{xspace}
\usepackage{cite}
\usepackage{enumitem}
\usepackage{url}
\usepackage{tikz}
\usepackage{color}
\usepackage{xcolor}
\usetikzlibrary{patterns}
\usepackage[section]{placeins}

\makeatletter 
\def\@cite#1#2{{\m@th\upshape\bfseries%
[{#1\if@tempswa{\m@th\upshape\mdseries, #2}\fi}]}}
\makeatother 

\theoremstyle{plain}
\newtheorem{thm}{Theorem}[section]
\newtheorem{cor}[thm]{Corollary}
\newtheorem{prop}[thm]{Proposition}
\newtheorem{lem}[thm]{Lemma}

\theoremstyle{definition}

\newtheorem{conj}[thm]{Conjecture}
\newtheorem{conv}[thm]{Convention}
\theoremstyle{remark}

\numberwithin{equation}{subsection}
\renewcommand{\bold}[1]{\medskip \noindent {\bf #1 }\nopagebreak}

\newcommand{\nc}{\newcommand}
\newcommand{\rnc}{\renewcommand}

\nc\bA{\mathbb{A}}
\nc\bB{\mathbb{B}}
\nc\bC{\mathbb{C}}
\nc\bD{\mathbb{D}}
\nc\bE{\mathbb{E}}
\nc\bF{\mathbb{F}}
\nc\bG{\mathbb{G}}
\nc\bH{\mathbb{H}}
\nc\bI{\mathbb{I}}
\nc{\bJ}{\mathbb{J}} 
\nc\bK{\mathbb{K}}
\nc\bL{\mathbb{L}}
\nc\bM{\mathbb{M}}
\nc\bN{\mathbb{N}}
\nc\bO{\mathbb{O}}
\nc\bP{\mathbb{P}}
\nc\bQ{\mathbb{Q}}
\nc\bR{\mathbb{R}}
\nc\bS{\mathbb{S}}
\nc\bT{\mathbb{T}}
\nc\bU{\mathbb{U}}
\nc\bV{\mathbb{V}}
\nc\bW{\mathbb{W}}
\nc\bY{\mathbb{Y}}
\nc\bX{\mathbb{X}}
\nc\bZ{\mathbb{Z}}
\nc\cA{\mathcal{A}}
\nc\cB{\mathcal{B}}
\nc\cC{\mathcal{C}}
\rnc\cD{\mathcal{D}}
\nc\cE{\mathcal{E}}
\nc\cF{\mathcal{F}}
\nc\cG{\mathcal{G}}
\rnc\cH{\mathcal{H}}
\nc\cI{\mathcal{I}}
\nc{\cJ}{\mathcal{J}} 
\nc\cK{\mathcal{K}}
\rnc\cL{\mathcal{L}}
\nc\cM{\mathcal{M}}
\nc\cN{\mathcal{N}}
\nc\cO{\mathcal{O}}
\nc\cP{\mathcal{P}}
\nc\cQ{\mathcal{Q}}
\rnc\cR{\mathcal{R}}
\nc\cS{\mathcal{S}}
\nc\cT{\mathcal{T}}
\nc\cU{\mathcal{U}}
\nc\cV{\mathcal{V}}
\nc\cW{\mathcal{W}}
\nc\cY{\mathcal{Y}}
\nc\cX{\mathcal{X}}
\nc\cZ{\mathcal{Z}}

\nc{\dmo}{\DeclareMathOperator}
\dmo{\Tw}{Twist}
\dmo{\CP}{Pres}
\rnc{\Re}{\operatorname{Re}}
\rnc{\Im}{\operatorname{Im}}
\rnc{\span}{\operatorname{span}}
\dmo{\rank}{rank}
\dmo{\End}{End}
\dmo{\Jac}{Jac}
\dmo{\Id}{Id}
\dmo{\lcm}{lcm}
\dmo{\Area}{Area}

\nc{\Tm}{Teichm\"uller\xspace}
\nc{\odd}{\cH^{\text{odd}}(4)}
\nc{\hyp}{\cH^{\text{hyp}}(4)}
\nc{\prym}{\tilde{\mathcal{Q}}(3,-1^3)}
\nc{\G}{GL^+(2,\bR)}

\begin{document}

\title[Orbit closures in $\odd$]{Classification of higher rank orbit closures in $\odd$}
%
\author[Aulicino]{David~Aulicino}
\address{\hspace{-0.5cm} Math\ Department\newline
University of Chicago\newline
5734 South University Avenue\newline
Chicago, IL 60637}
\email{aulicino@math.uchicago.edu}
\author[Nguyen]{Duc-Manh~Nguyen}
\address{\hspace{-0.5cm} IMB Bordeaux-Universit\'e de Bordeaux \newline
351, Cours de la Lib\'eration \newline
33405 Talence Cedex \newline FRANCE}
\email{duc-manh.nguyen@math.u-bordeaux1.fr}
\author[Wright]{Alex~Wright}
\address{\hspace{-0.5cm} Math\ Department\newline
University of Chicago\newline
5734 South University Avenue\newline
Chicago, IL 60637}
\email{alexmwright@gmail.com}
%

\begin{abstract}
The moduli space of genus 3 translation surfaces with a single zero has two connected components. We show that in the odd connected component $\odd$ the only $\G$ orbit closures are closed orbits, the Prym locus $\prym$, and $\odd$. 

Together with work of Matheus-Wright, this implies that there are only finitely many non-arithmetic closed orbits (\Tm curves) in $\odd$ outside of the Prym locus. 
\end{abstract}

\maketitle
\thispagestyle{empty}


\section{Introduction}\label{S:intro}

An affine invariant submanifold is a subset of a stratum of translation surfaces that is defined locally by real linear homogeneous equations in period coordinates (see Section 2 for a more precise definition). We say that an affine invariant submanifold $\cN$ is \emph{rank 1} if the only way to deform a translation surface in $\cN$ so that the deformed surface is also in $\cN$ is to combine the $\G$ action with  deformations that fix absolute periods. Otherwise we say $\cN$ is \emph{higher rank}.

In minimal strata $\cH(2g-2)$, rank 1 affine invariant submanifolds are closed orbits. The purpose of this paper is to show

\begin{thm}\label{T:main}
The only proper higher rank affine invariant submanifold of $\odd$ is the Prym locus $\prym$.  
\end{thm}

By proper, we mean not equal to a connected component of a stratum. Together with recent work Eskin-Mirzakhani-Mohammadi, discussed below, Theorem \ref{T:main} gives that in $\odd$, every $\G$-orbit is either closed, dense in $\prym$, or dense in $\odd$. 

The Prym locus $\prym$ is the set of all holonomy double covers of genus 1 quadratic differentials with 1 zero of order 3 and 3 simple poles. (The holonomy double cover of a quadratic differential is a possibly branched double cover equipped with an abelian differential whose square is the pullback of the quadratic differential.) Equivalently, the Prym locus is the set of all translation surfaces in the minimal stratum in genus 3 that admit a flat involution with derivative -1 and 4 fixed points. It follows from work of Lanneau \cite{Lconn} (and, without too much work, the methods in this paper) that $\prym$ is connected. 

Below we will indicate an application of Theorem \ref{T:main} to finiteness of Teichm\"uller curves, and explain that it provides evidence for a conjecture of 
Mirzakhani. Theorem \ref{T:main} confirms this conjecture in the special case of $\odd$, which is the first connected component of a stratum analyzed so far that actually contains a higher rank affine invariant submanifold.  Now we give the context for our work. 

Ten years ago, McMullen and Calta independently found and studied infinitely many rank 1 affine invariant submanifolds in each of the two strata in genus 2 \cite{Ca,Mc} (they did not use this terminology). Prior to this, it had been known that almost every translation surface has dense orbit \cite{Ma2, V2}, but the only examples of proper orbit closures not coming from covering constructions were the non-arithmetic Teichm\"uller curves of Veech and Ward \cite{V,W}.


McMullen went on to classify orbit closures in genus two \cite{McM:spin, Mc4, Mc5}. A very particular consequence of his work is that no proper higher rank affine invariant submanifolds exist in genus 2.   

Some generalizations of McMullen's techniques were made in genus 3 \cite{HLM-AY, HLM-Q, N}, showing that some especially interesting translation surfaces have dense orbits. This made it reasonable to conjecture that few truly new orbit closures exist in genus greater than 2.   

Recently, progress has been obtained in arbitrary genus. 

\begin{thm}[Eskin-Mirzakhani-Mohammadi \cite{EM, EMM}]\label{T:EMM}
Every $\G$ orbit closure of a translation surface is an affine invariant submanifold.
\end{thm}

This result and its proof give almost no  information that might be useful in classifying affine invariant submanifolds.\footnote{Added in proof: After the completion of the present work, Filip proved that affine invariant submanifolds are varieties \cite{Fi1, Fi2}. This work \emph{does} give a great deal of  algebro-geometric information; however it is at this moment still unclear how to apply it to the classification problem.}

The linear equations that locally define an affine invariant submanifold may be taken to have coefficients in a number field \cite{Wfield}. However, the earlier work mentioned above suggests that affine invariant submanifolds  are incredibly special: most sets of linear equations with coefficients in a number field should \emph{not} locally define part of an affine invariant submanifold. The most precise such conjecture in this direction is due to Mirzakhani. We paraphrase it here.  

\begin{conj}[Mirzakhani]
Any higher rank affine invariant submanifold $\cN$ is either a stratum, or is ``not primitive" in the sense that every translation surface covers a quadratic differential of lower genus. 
\end{conj}  

In the non-primitive case, $\cN$ should arise from a ``covering construction", but it is an open problem to determine exactly which sorts of such constructions are possible. 

The third author showed that in rank 1 affine invariant submanifolds, all translation surfaces are completely periodic \cite{Wcyl}. This is one reason why rank 1 is very special. The Cylinder Deformation Theorem of \cite{Wcyl} (Theorem \ref{T:CDT} below) supports the conjecture that only few affine invariant submanifolds of higher rank exist. 

Our work builds upon recent work of the last two authors on the hyperelliptic connected component of the minimal stratum in genus 3.

\begin{thm}[Nguyen-Wright \cite{NW}]\label{T:hyp}
There are no proper higher rank affine invariant submanifolds in $\hyp$. 
\end{thm} 



%
%

The analysis of $\odd$ is more complicated than $\hyp$, due to the existence of a greater number of cylinder diagrams. Consequently we have developed new techniques  capable of handling many cases at once.  However, the main novelty of this paper is to start with a rank 2 affine invariant submanifold $\cN$, and to establish sufficient symmetry in the flat geometry to show that any $M\in \cN$ must cover a surface of lower genus. 

An immediate corollary of our work, together with recent work of Matheus-Wright \cite{MW}, is 

\begin{thm}\label{T:main}
There are only finitely many non-arithmetic Teichm\"uller curves in $\odd$ outside of the Prym locus. 
\end{thm}

McMullen has constructed infinitely many non-arithmetic Teichm\"uller curves in the Prym locus \cite{Mc2}, which have been studied by M\"oller \cite{Mprym} and Lanneau-Nguyen \cite{LNprym}. There are at present no known non-arithmetic Teichm\"uller curves in $\odd$ outside of the Prym locus. For the context of Theorem \ref{T:main}, including work of Bainbridge, Bouw, Calta, McMullen, M\"oller, and Veech \cite{BaM, Ca, Mc, Mc4, M, M2, M3, V}, see \cite{MW} and the introduction to \cite{NW}. A list of all known non-arithmetic Teichm\"uller curves can be found in the introduction to \cite{W2}.

\bold{Acknowledgments.} We thank Alex Eskin and Maryam Mirzakhani for helpful discussions about affine invariant submanifolds. We are  grateful to Samuel Leli\`evre for creating a document on cylinder diagrams in $\cH(4)$, now available as Appendix C in \cite{MMY}.

The first author was partially supported by the National Science Foundation under Award No. DMS - 1204414.  

\section{Affine invariant submanifolds and cylinder deformations}\label{S:intro}

This section reviews relevant background. 

Every stratum $\cH$ admits a finite orbifold cover $\cH'$ to which the $\G$ action lifts, such that $\cH'$ is a fine moduli space. (Passing to such a cover allows us to assume that $\cH'$ has no orbifold points and is a fine moduli space. This avoids any subtleties in the discussion of local coordinates below. The cover must typically be finite to ensure that, for example, the lift of an orbit  is a finite union of orbit.) For example, $\cH'$ can be taken to be the moduli space of translation surfaces in $\cH$ together with a choice of level 3 structure. Anyone not familiar with this issue is advised to pretend $\cH'=\cH$. In minimal strata, one can indeed take $\cH=\cH'$, and hence we will soon ignore the distinction.

 Near any $M\in \cH'$, local coordinates may be defined as follows. Pick a basis $\gamma_1, \ldots, \gamma_n$ for the relative homology group $H_1(M, \Sigma; \bZ)$, where $\Sigma$ is the set of zeros of $M$.  The local coordinates are 
$$M'=(X', \omega') \mapsto \left( \int_{\gamma_i} \omega'\right) \in \bC^n.$$
An \emph{affine invariant submanifold} in $\cH'$ is a submanifold that is defined in these local coordinates by homogeneous linear equations with real coefficients. An affine invariant submanifold of $\cH$ is the image of one in $\cH'$ under the covering map. 

Let us fix some standard notation. A \emph{cylinder} on a translation surface is the image of an isometric injection from an open Euclidean cylinder $ (\bR/c\bZ)\times (0,h)$. The \emph{height} of the cylinder is $h$, the \emph{circumference} is $c$, and the \emph{modulus} is $h/c$.  We will always assume that cylinders on translation surfaces are maximal, so that their height cannot be increased. A core curve of a cylinder is defined to be the image of any circle $ (\bR/c\bZ)\times \{x_0\}$.  

A saddle connection on a translation surface is a line segment connecting two zeros. The holonomy of a relative homology class on $(X,\omega)$ is the integral of $\omega$ over any cycle representing this class. The holonomy of a saddle connection is the holonomy of the associated relative homology class; this is a vector in $\bC$. For more basic definitions, see \cite{MT, Z}.

Now suppose that $\cN$ is an affine invariant submanifold, and take $M\in \cN$. Suppose that $X$ is a cylinder on $M$. On small deformations $M'$ of $M$ (i.e., for all $M'$ in a sufficiently small simply connected neighborhood of $M$), the cylinder $X$ persists. That is, there is a corresponding cylinder $X'$ on $M'$ whose height, circumference, and direction are all close to those of $X$, and such that  the core curves of $X$ and $X'$ represent the same homology class. Sometimes we may speak of $X'$ as being ``the cylinder $X$ on $M'$."

Two cylinders $X$ and $Y$ on $M\in \cN$ are said to be \emph{$\cN$-parallel} if they are parallel, and remain parallel on all deformations of $M$. The deformations are assumed to be small, so that $X$ and $Y$ persist. Two cylinders $X$ and $Y$ are $\cN$-parallel if and only if there is a constant $c\in \bR$ such that on all deformations in $\cN$, the holonomy of the core curve of $Y$ is $c$ times that of $X$. In other words, two cylinders on $M\in \cN$ are $\cN$-parallel if and only if one of the linear equations defining $\cN$ in local period coordinates makes it so. The relation of being $\cN$-parallel is an equivalence relation, and when we speak of an equivalence class of a cylinder, we mean the set of all cylinders $\cN$-parallel to it.

Define the matrices
$$u_t = \left(\begin{array}{cc} 1 & t\\0 & 1 \end{array}\right), \quad a_s = \left(\begin{array}{cc} 1 & 0\\0 & e^s \end{array}\right),\quad r_\theta = \left(\begin{array}{cc} \cos \theta  & -\sin \theta  \\\sin \theta  & \cos \theta  \end{array}\right).$$
Let $\cC$ be a collection of parallel cylinders on a translation surface $M$. Suppose they are all of angle $\theta \in[0,\pi)$, meaning that plus or minus the holonomy of all the core curves have angle $\theta$ measured counterclockwise from the  positive real direction. Define $a_s^\cC(u_t^\cC(M))$ to be the translation surface obtained by applying $r_{-\theta}$ to $M$, then applying the matrix $a_s u_t$ to the images of the cylinders in $\cC$, and then applying $r_\theta$. 

\begin{thm}[The Cylinder Deformation Theorem \cite{Wcyl}]\label{T:CDT}
Suppose that $\cC$ is an equivalence class of $\cN$-parallel cylinders on $M\in \cN$ at angle $\theta$. Then for all $s, t\in \bR$, the surface $a_s^\cC(u_t^\cC(M))\in \cN$. 
\end{thm}

We call $a_s^\cC$ the \emph{cylinder stretch}, and $u_t^\cC$ the \emph{cylinder shear}.

If $\cC=\{C\}$ consists of a single cylinder on $M\in \cN$, and $a_s^\cC(u_t^\cC(M))\in \cN$ for all $s$ and $t$, then we will say that $C$ is \emph{$\cN$-free}. If $\cN$ is clear from context, we will call $C$ \emph{free}.  A corollary of the Cylinder Deformation Theorem, which we shall use frequently, is that if a cylinder is not $\cN$-parallel to any other then it must be free. 

Suppose that $\cE$ is an equivalence class of $\cN$-parallel cylinders on $M$, and $X$ is a cylinder on $M$. Denote by $P(X,\cE)$ the \emph{proportion of $X$ which lies in $\cE$}. That is, 
$$P(X,\cE)= \frac{\Area(X\cap (\cup_{C\in \cE} C))}{\Area(X)}.$$ 

An elementary consequence of the Cylinder Deformation Theorem is 

\begin{prop}[Nguyen-Wright \cite{NW}]\label{P:P}
Let $X$ and $Y$ be $\cN$-parallel cylinders on $M$, and let $\cE$ be an equivalence class of $\cN$-parallel cylinders on $M$. Then $P(X, \cE)=P(Y,\cE)$. 
\end{prop}
This proposition is one of the main tools in this paper. 

Let $T_M(\cN)$ denote the tangent space to $\cN$ at $M\in \cN$. This is naturally a subspace of $H^1(M,\Sigma; \bC)$ which is defined by linear equations with real coefficients. Therefore we can write $T_M(\cN)=\bC\otimes_{\bR} T^{\bR}_M(\cN)$, where $T^{\bR}_M(\cN) \subset H^1(M,\Sigma,\bR)$. Let $p:H^1(M,\Sigma; \bR)\to H^1(M; \bR)$ be the natural projection. Some of the results we cite in this paper, including Theorem \ref{T:EMM}, use

\begin{thm}[Avila-Eskin-M\"oller \cite{AEM}]
Let $\cN$ be an affine invariant submanifold, and let $M\in \cN$. Then $p(T^{\bR}_M(\cN))$ is symplectic. 
\end{thm}

In \cite{Wcyl} the third author first suggested that the number $\frac12 \dim_{\bR} p(T^{\bR}_M(\cN))$ be called the (cylinder) \emph{rank of $\cN$}. Any proper affine invariant submanifold of $\cH(4)$ has rank 1 or 2.

\section{Periodic directions in $\odd$}\label{S:intro}

This section discusses the combinatorics of periodic translation surfaces in $\odd$, setting the foundation for our analysis. 

A translation surface is said to be periodic in some direction if the surface is the union of saddle connections and closed trajectories in this direction. After rotation, every periodic direction gives a horizontally periodic translation surface. 

In this paper, we consider horizontally periodic translation surfaces to be (combinatorially) equivalent if there is a homeomorphism between them, taking positively oriented horizontal leaves to positively oriented horizontal leaves. Thus the lengths of the horizontal saddle connections and the circumferences of horizontal cylinders can be changed, and the horizontal cylinders can be individually sheared and stretched, to yield  equivalent horizontally periodic translation surfaces. 

Equivalence classes under this notion of equivalence are called \emph{cylinder diagrams}. We emphasize that cylinder diagrams do not contain any metric information: they indicate only the number of cylinders, and the cyclic order of saddle connections on the top and bottom of each cylinder. Cylinder diagrams do however contain an orientation, so that ``left" and ``up" have meaning in these diagrams. Frequently, some saddle connections are labeled to indicate gluings, but the exact labels used are not considered part of the data of the cylinder diagram.

Some number of years ago, Samuel Leli\`evre used Sage to enumerate the 22 cylinder diagrams in $\cH(4)$, and produced a beautiful document illustrating all of these possibilities. This document helped to inspire and facilitate this work. Recently it has appeared as Appendix C in \cite{MMY}. 

Leli\`evre gives 7 cylinder diagrams in $\odd$ with 3 cylinders, and 7 with 2 cylinders. However, there is some additional symmetry which can be exploited. There is a $\bZ_2\times \bZ_2$ action on the set of cylinder diagrams, where the two generators act by reflection in the $x$-axis and in the $y$-axis respectively. 

These symmetries are most easily thought of by drawing the cylinder diagram on the page, and then reflecting this picture. This is easily shown to be well defined. The effect of ``reflection in the $x$-axis" is to switch the roles of top and bottom saddle connections on each cylinder. The effect of ``reflection in the $y$-axis" is to  reverse the cyclic order of saddle connections on each boundary component of each cylinder.  

\begin{prop}\label{P:8diagrams}
Every cylinder diagram in $\odd$ with at least 2 cylinders is equal to one of those given in Figure \ref{F:8diagrams}, possibly after horizontal or vertical reflection (or both). 
\end{prop} 

We have checked this proposition using Leli\`evre's list, and also by an independent analysis without the use of computers, which is tedious but straightforward. 

A surface in $\cH(4)$ can have at most 3 parallel cylinders in any given direction. 

Proposition \ref{P:8diagrams} can be restated as follows: Let $M\in \odd$ be horizontally periodic with at least 2 horizontal cylinders. Then possibly after horizontal or vertical reflection (or both), the cylinder diagram of $M$ is equal to one of the 8 diagrams listed in  Figure \ref{F:8diagrams}. A horizontal or vertical reflection of a translation surface is defined to be an affine self map whose derivative is the desired reflection on $\bR^2$. 

To acknowledge the possible need to vertically and horizontally reflect a translation surface, we will write the phrase ``up to symmetry". It is worth noting that reflecting a translation surface may change its orbit closure; but nonetheless such reflections will not affect our arguments. 

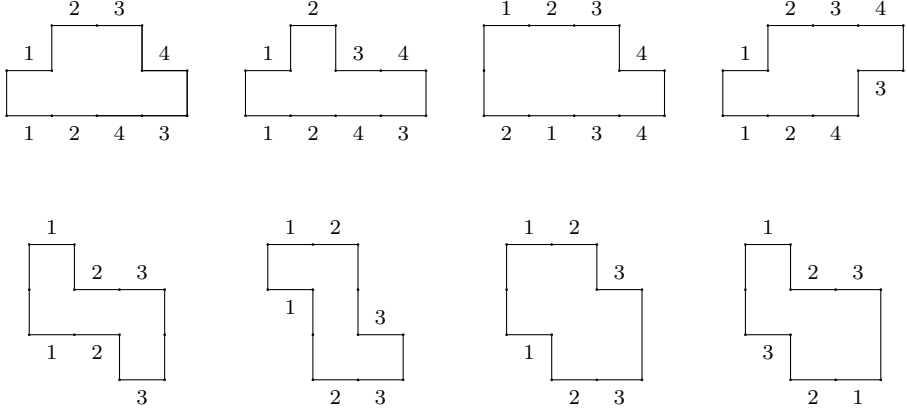
\begin{figure}[h!]
\begin{minipage}[t]{0.24\linewidth}
\centering
\begin{tikzpicture}[scale=0.30]
\draw (0,0) -- (0,2) -- (2,2) -- (2,4) -- (4,4) -- (6,4) -- (6,2) -- (8,2) -- (8,0) -- (6,0) -- (4,0) -- (2,0) --cycle;
\draw (4,0) -- (8,0) -- (8,2) -- (6,2) -- (6,4) -- (4,4);
\foreach \x in {(0,0), (2,0), (4,0), (6,0), (8,0), (8,2), (6,2), (6,4), (4,4), (2,4), (2,2), (0,2)} \draw \x circle (1pt);
\draw(1,2) node[above] {\tiny 1};
\draw(3,4) node[above] {\tiny 2};
\draw(5,4) node[above] {\tiny 3};
\draw(7,2) node[above] {\tiny 4};
\draw(1,0) node[below] {\tiny 1};
\draw(3,0) node[below] {\tiny 2};
\draw(5,0) node[below] {\tiny 4};
\draw(7,0) node[below] {\tiny 3};
\end{tikzpicture}
\end{minipage}
\begin{minipage}[t]{0.24\linewidth}
\centering
\begin{tikzpicture}[scale=0.30]
\draw (0,0)--(0,2)--(2,2)--(2,4)--(4,4)--(4,2)--(6,2)--(8,2)--(8,0)--(6,0)--(4,0)--(2,0)--cycle;
\foreach \x in {(0,0),(0,2),(2,2),(2,4),(4,4),(4,2),(6,2),(8,2),(8,0),(6,0),(4,0),(2,0)} \draw \x circle (1pt);
\draw(1,2) node[above] {\tiny 1};
\draw(3,4) node[above] {\tiny 2};
\draw(5,2) node[above] {\tiny 3};
\draw(7,2) node[above] {\tiny 4};
\draw(1,0) node[below] {\tiny 1};
\draw(3,0) node[below] {\tiny 2};
\draw(5,0) node[below] {\tiny 4};
\draw(7,0) node[below] {\tiny 3};
\end{tikzpicture}
\end{minipage}
\begin{minipage}[b]{0.24\linewidth}
\centering
\begin{tikzpicture}[scale=0.30]
\draw (0,0)--(0,2)--(0,4)--(2,4)--(4,4)--(6,4)--(6,2)--(8,2)--(8,0)--(6,0)--(4,0)--(2,0)--cycle;
\foreach \x in {(0,0),(0,2),(0,4),(2,4),(4,4),(6,4),(6,2),(8,2),(8,0),(6,0),(4,0),(2,0)} \draw \x circle (1pt);
\draw(1,4) node[above] {\tiny 1};
\draw(3,4) node[above] {\tiny 2};
\draw(5,4) node[above] {\tiny 3};
\draw(7,2) node[above] {\tiny 4};
\draw(1,0) node[below] {\tiny 2};
\draw(3,0) node[below] {\tiny 1};
\draw(5,0) node[below] {\tiny 3};
\draw(7,0) node[below] {\tiny 4};
\end{tikzpicture}
\end{minipage}
\begin{minipage}[b]{0.24\linewidth}
\centering
\begin{tikzpicture}[scale=0.30]
\draw (0,0)--(0,2)--(2,2)--(2,4)--(4,4)--(6,4)--(8,4)--(8,2)--(6,2)--(6,0)--(4,0)--(2,0)--cycle;
\foreach \x in {(0,0),(0,2),(2,2),(2,4),(4,4),(6,4),(8,4),(8,2),(6,2),(6,0),(4,0),(2,0)} \draw \x circle (1pt);
\draw(1,2) node[above] {\tiny 1};
\draw(3,4) node[above] {\tiny 2};
\draw(5,4) node[above] {\tiny 3};
\draw(7,4) node[above] {\tiny 4};
\draw(1,0) node[below] {\tiny 1};
\draw(3,0) node[below] {\tiny 2};
\draw(5,0) node[below] {\tiny 4};
\draw(7,2) node[below] {\tiny 3};
\end{tikzpicture}
\end{minipage}

\vspace{0.75cm}

\begin{minipage}[t]{0.24\linewidth}
\centering
\begin{tikzpicture}[scale=0.30]
\draw (0,2)--(0,4)--(0,6)--(2,6)--(2,4)--(4,4)--(6,4)--(6,2)--(6,0)--(4,0)--(4,2)--(2,2)--cycle;
\foreach \x in {(0,2),(0,4),(0,6),(2,6),(2,4),(4,4),(6,4),(6,2),(6,0),(4,0),(4,2),(2,2)} \draw \x circle (1pt);
\draw(1,6) node[above] {\tiny 1};
\draw(3,4) node[above] {\tiny 2};
\draw(5,4) node[above] {\tiny 3};
\draw(1,2) node[below] {\tiny 1};
\draw(3,2) node[below] {\tiny 2};
\draw(5,0) node[below] {\tiny 3};
\end{tikzpicture}
\end{minipage}
\begin{minipage}[b]{0.24\linewidth}
\centering
\begin{tikzpicture}[scale=0.30]
\draw (0,4)--(0,6)--(2,6)--(4,6)--(4,4)--(4,2)--(6,2)--(6,0)--(4,0)--(2,0)--(2,2)--(2,4)--cycle;
\foreach \x in {(0,4),(0,6),(2,6),(4,6),(4,4),(4,2),(6,2),(6,0),(4,0),(2,0),(2,2),(2,4)} \draw \x circle (1pt);
\draw(1,6) node[above] {\tiny 1};
\draw(3,6) node[above] {\tiny 2};
\draw(5,2) node[above] {\tiny 3};
\draw(1,4) node[below] {\tiny 1};
\draw(3,0) node[below] {\tiny 2};
\draw(5,0) node[below] {\tiny 3};
\end{tikzpicture}
\end{minipage}
\begin{minipage}[b]{0.24\linewidth}
\centering
\begin{tikzpicture}[scale=0.30]
\draw (0,2)--(0,4)--(0,6)--(2,6)--(4,6)--(4,4)--(6,4)--(6,2)--(6,0)--(4,0)--(2,0)--(2,2)--cycle;
\foreach \x in {(0,2),(0,4),(0,6),(2,6),(4,6),(4,4),(6,4),(6,0),(4,0),(2,0),(2,2)} \draw \x circle (1pt);
\draw(1,6) node[above] {\tiny 1};
\draw(3,6) node[above] {\tiny 2};
\draw(5,4) node[above] {\tiny 3};
\draw(1,2) node[below] {\tiny 1};
\draw(3,0) node[below] {\tiny 2};
\draw(5,0) node[below] {\tiny 3};
\end{tikzpicture}
\end{minipage}
\begin{minipage}[t]{0.24\linewidth}
\centering
\begin{tikzpicture}[scale=0.30]
\draw (0,2)--(0,4)--(0,6)--(2,6)--(2,4)--(4,4)--(6,4)--(6,2)--(6,0)--(4,0)--(2,0)--(2,2)--cycle;
\foreach \x in {(0,2),(0,4),(0,6),(2,6),(2,4),(4,4),(6,4),(6,0),(4,0),(2,0),(2,2)} \draw \x circle (1pt);
\draw(1,6) node[above] {\tiny 1};
\draw(3,4) node[above] {\tiny 2};
\draw(5,4) node[above] {\tiny 3};
\draw(1,2) node[below] {\tiny 3};
\draw(3,0) node[below] {\tiny 2};
\draw(5,0) node[below] {\tiny 1};
\end{tikzpicture}
\end{minipage}
\caption{Up to symmetry, there are $8$ cylinder diagrams in $\odd$ with either 2 or 3 cylinders. In all such pictures in this paper, opposite vertical edges of polygons are identified, and horizontal edge labeling indicate additional edge identifications.}
\label{F:8diagrams}
\end{figure}

\section{Getting three cylinders}\label{S:intro}

In this section we show that any rank 2 affine invariant submanifold of $\odd$ has a horizontally periodic surface with three horizontal cylinders. 

\begin{thm}[Wright \cite{Wcyl}]\label{T:manyC}
Let $\cN$ be an affine invariant submanifold. Then $\cN$ contains a horizontally periodic translation surface with at least $\rank(\cN)$ many horizontal cylinders. 
\end{thm}

\begin{thm}[Smillie-Weiss \cite{SW2}]\label{T:SW}
The $u_t$ orbit closure of every translation surface contains a horizontally periodic translation surface. 
\end{thm}

\begin{prop}[Nguyen-Wright \cite{NW}]\label{P:NWfree}
If every surface in an affine invariant submanifold $\cN$ contains at most $\rank(\cN)$ horizontal cylinders, then every cylinder on every surface in $\cN$ is free. 
\end{prop} 

\begin{prop}\label{P:3cyl}
Let $\cN$ be a rank 2 affine invariant submanifold of $\odd$. Then $\cN$ contains a horizontally periodic surface with 3 horizontal cylinders. 
\end{prop}

\begin{proof}
By Theorem \ref{T:manyC}, there is a horizontally periodic surface $M$ in $\cN$ with at least 2 horizontal cylinders. If $M$ has 3 horizontal cylinders the result is proved, so we will assume that $M$ has only 2 horizontal cylinders. Thus by Proposition \ref{P:NWfree}, every cylinder on every surface in $\cN$ is free.

By Proposition \ref{P:8diagrams}, $M$ has one of the four 2-cylinder diagrams in Figure \ref{F:8diagrams} (up to symmetry). Figure \ref{F:2cyls} shows that, in each of these four cases, the 2 horizontal cylinders on $M$ can be sheared to give a surface $M'$ with 2 vertical cylinders whose union is not all of $M'$.  

Rotate $M'$ by $\frac\pi2$ to give a surface $M''$ with two horizontal cylinders whose union is not the whole surface. By Theorem \ref{T:SW}, there is a horizontally periodic translation surface $M'''$ in the $u_t$ orbit closure of $M''$. This $M'''$ must have  3 horizontal cylinders. 
\end{proof}

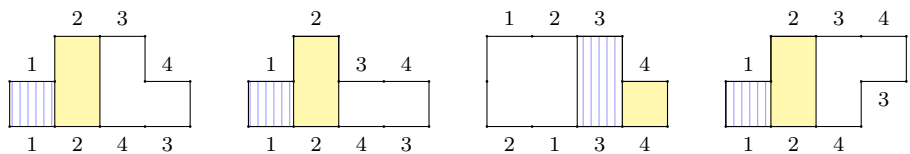
\begin{figure}[h!]
\begin{minipage}[t]{0.24\linewidth}
\centering
\begin{tikzpicture}[scale=0.30]
\draw [pattern=vertical lines, pattern color=blue!40] (0,0) rectangle (2,2);
\draw [fill=yellow!40] (2,0) rectangle (4,4);
\draw (4,0) -- (8,0) -- (8,2) -- (6,2) -- (6,4) -- (4,4);
\foreach \x in {(0,0), (2,0), (4,0), (6,0), (8,0), (8,2), (6,2), (6,4), (4,4), (2,4), (2,2), (0,2)} \draw \x circle (1pt);
\draw(1,2) node[above] {\tiny 1};
\draw(3,4) node[above] {\tiny 2};
\draw(5,4) node[above] {\tiny 3};
\draw(7,2) node[above] {\tiny 4};
\draw(1,0) node[below] {\tiny 1};
\draw(3,0) node[below] {\tiny 2};
\draw(5,0) node[below] {\tiny 4};
\draw(7,0) node[below] {\tiny 3};
\end{tikzpicture}
\end{minipage}
\begin{minipage}[t]{0.24\linewidth}
\centering
\begin{tikzpicture}[scale=0.30]
\draw [pattern=vertical lines, pattern color=blue!40] (0,0) rectangle (2,2);
\draw [fill=yellow!40] (2,0) rectangle (4,4);
\draw (0,0)--(0,2)--(2,2)--(2,4)--(4,4)--(4,2)--(6,2)--(8,2)--(8,0)--(6,0)--(4,0)--(2,0)--cycle;
\foreach \x in {(0,0),(0,2),(2,2),(2,4),(4,4),(4,2),(6,2),(8,2),(8,0),(6,0),(4,0),(2,0)} \draw \x circle (1pt);
\draw(1,2) node[above] {\tiny 1};
\draw(3,4) node[above] {\tiny 2};
\draw(5,2) node[above] {\tiny 3};
\draw(7,2) node[above] {\tiny 4};
\draw(1,0) node[below] {\tiny 1};
\draw(3,0) node[below] {\tiny 2};
\draw(5,0) node[below] {\tiny 4};
\draw(7,0) node[below] {\tiny 3};
\end{tikzpicture}
\end{minipage}
\begin{minipage}[b]{0.24\linewidth}
\centering
\begin{tikzpicture}[scale=0.30]
\draw [pattern=vertical lines, pattern color=blue!40] (4,0) rectangle (6,4);
\draw [fill=yellow!40] (6,0) rectangle (8,2);
\draw (0,0)--(0,2)--(0,4)--(2,4)--(4,4)--(6,4)--(6,2)--(8,2)--(8,0)--(6,0)--(4,0)--(2,0)--cycle;
\foreach \x in {(0,0),(0,2),(0,4),(2,4),(4,4),(6,4),(6,2),(8,2),(8,0),(6,0),(4,0),(2,0)} \draw \x circle (1pt);
\draw(1,4) node[above] {\tiny 1};
\draw(3,4) node[above] {\tiny 2};
\draw(5,4) node[above] {\tiny 3};
\draw(7,2) node[above] {\tiny 4};
\draw(1,0) node[below] {\tiny 2};
\draw(3,0) node[below] {\tiny 1};
\draw(5,0) node[below] {\tiny 3};
\draw(7,0) node[below] {\tiny 4};
\end{tikzpicture}
\end{minipage}
\begin{minipage}[b]{0.24\linewidth}
\centering
\begin{tikzpicture}[scale=0.30]
\draw [pattern=vertical lines, pattern color=blue!40] (0,0) rectangle (2,2);
\draw [fill=yellow!40] (2,0) rectangle (4,4);
\draw (0,0)--(0,2)--(2,2)--(2,4)--(4,4)--(6,4)--(8,4)--(8,2)--(6,2)--(6,0)--(4,0)--(2,0)--cycle;
\foreach \x in {(0,0),(0,2),(2,2),(2,4),(4,4),(6,4),(8,4),(8,2),(6,2),(6,0),(4,0),(2,0)} \draw \x circle (1pt);
\draw(1,2) node[above] {\tiny 1};
\draw(3,4) node[above] {\tiny 2};
\draw(5,4) node[above] {\tiny 3};
\draw(7,4) node[above] {\tiny 4};
\draw(1,0) node[below] {\tiny 1};
\draw(3,0) node[below] {\tiny 2};
\draw(5,0) node[below] {\tiny 4};
\draw(7,2) node[below] {\tiny 3};
\end{tikzpicture}
\end{minipage}
\caption{In any horizontally periodic translation surface in $\odd$ with 2 horizontal cylinders, it is possible to twist the horizontal cylinders so that there are 2 vertical cylinders whose union is not the whole surface.}
\label{F:2cyls}
\end{figure}

\section{Setting up a case by case analysis}\label{S:intro}

The section outlines the structure of our analysis, which follows that of \cite{NW} up to a point. The cases given in this section are the topic of the remainder of the paper. 

\begin{prop}[Nguyen-Wright \cite{NW}]\label{P:3free}
Let $\cN$ be an affine invariant submanifold of $\cH(4)$. If $\cN$ contains a horizontally periodic surface with $3$ free horizontal cylinders, then $\cN$ is equal to a connected component of $\cH(4)$.
\end{prop}


Let $\cN$ be a rank 2 affine invariant submanifold of $\odd$. Proposition \ref{P:3cyl} gives that there is an $M\in \cN$ with three horizontal cylinders. Proposition \ref{P:3free} gives that these three horizontal cylinders cannot all be free. On the other hand, \cite[Section 8]{Wcyl} shows that not all three horizontal cylinders on $M$ can be $\cN$-equivalent to each other. (This is explained in more detail in \cite{NW}.) Hence, we get 

\begin{prop}\label{P:onefree}
Let $\cN$ be an affine invariant submanifold of $\cH(4)$, and suppose $M\in \cN$ has three horizontal cylinders. Then one of these three horizontal cylinders is free, and the other two are $\cN$-parallel to each other. 
\end{prop}

We now have a number of not necessarily disjoint cases. By Proposition \ref{P:3cyl}, $\cN$ contains a horizontally periodic surface whose cylinder diagram is equal to one of the four 3-cylinder diagrams in Figure \ref{F:8diagrams}, up to symmetry. On this surface, one of the three cylinders is free. Thus we have 12 cases: 3 choices of a free cylinder in each of the four 3-cylinder diagrams.  In Figure \ref{F:O1-4} these cases are enumerated.

\begin{figure}[htb]
\begin{minipage}[t]{0.24\linewidth}
\centering
\begin{tikzpicture}[scale=0.30]
\draw [fill=green!40] (0,4) rectangle (2,6);
\draw [fill=green!40] (4,0) rectangle (6,2);
\draw (0,2)--(0,4)--(0,6)--(2,6)--(2,4)--(4,4)--(6,4)--(6,2)--(6,0)--(4,0)--(4,2)--(2,2)--cycle;
\foreach \x in {(0,2),(0,4),(0,6),(2,6),(2,4),(4,4),(6,4),(6,2),(6,0),(4,0),(4,2),(2,2)} \draw \x circle (1pt);
\draw(1,6) node[above] {\tiny 1};
\draw(3,4) node[above] {\tiny 2};
\draw(5,4) node[above] {\tiny 3};
\draw(1,2) node[below] {\tiny 1};
\draw(3,2) node[below] {\tiny 2};
\draw(5,0) node[below] {\tiny 3};
\draw(1,5) node {\tiny A};
\draw(3,3) node {\tiny B};
\draw(5,1) node {\tiny C};
\draw(-1.5,3) node {\small (O1)};
\end{tikzpicture}
\end{minipage}
\begin{minipage}[b]{0.24\linewidth}
\centering
\begin{tikzpicture}[scale=0.30]
\draw [pattern=horizontal lines, pattern color=red!40] (2,0) rectangle (6,2);
\draw [pattern=horizontal lines, pattern color=red!40] (0,4) rectangle (4,6);
\draw (0,4)--(0,6)--(2,6)--(4,6)--(4,4)--(4,2)--(6,2)--(6,0)--(4,0)--(2,0)--(2,2)--(2,4)--cycle;
\foreach \x in {(0,4),(0,6),(2,6),(4,6),(4,4),(4,2),(6,2),(6,0),(4,0),(2,0),(2,2),(2,4)} \draw \x circle (1pt);
\draw(1,6) node[above] {\tiny 1};
\draw(3,6) node[above] {\tiny 2};
\draw(5,2) node[above] {\tiny 3};
\draw(1,4) node[below] {\tiny 1};
\draw(3,0) node[below] {\tiny 2};
\draw(5,0) node[below] {\tiny 3};
\draw(2,5) node {\tiny A};
\draw(3,3) node {\tiny B};
\draw(4,1) node {\tiny C};
\draw(-1.5,3) node {\small (O2)};
\end{tikzpicture}
\end{minipage}
\begin{minipage}[b]{0.24\linewidth}
\centering
\begin{tikzpicture}[scale=0.30]
\draw [fill=green!40] (2,0) rectangle (6,2);
\draw [fill=green!40] (0,4) rectangle (4,6);
\draw (0,2)--(0,4)--(0,6)--(2,6)--(4,6)--(4,4)--(6,4)--(6,2)--(6,0)--(4,0)--(2,0)--(2,2)--cycle;
\foreach \x in {(0,2),(0,4),(0,6),(2,6),(4,6),(4,4),(6,4),(6,0),(4,0),(2,0),(2,2)} \draw \x circle (1pt);
\draw(1,6) node[above] {\tiny 1};
\draw(3,6) node[above] {\tiny 2};
\draw(5,4) node[above] {\tiny 3};
\draw(1,2) node[below] {\tiny 1};
\draw(3,0) node[below] {\tiny 2};
\draw(5,0) node[below] {\tiny 3};
\draw(2,5) node {\tiny A};
\draw(3,3) node {\tiny B};
\draw(4,1) node {\tiny C};
\draw(-1.5,3) node {\small (O3)};
\end{tikzpicture}
\end{minipage}
\begin{minipage}[t]{0.24\linewidth}
\centering
\begin{tikzpicture}[scale=0.30]
\draw [pattern=horizontal lines, pattern color=red!40] (0,4) rectangle (2,6);
\draw [fill=green!40] (2,0) rectangle (6,2);
\draw [pattern = dots, pattern color=black!40] (0,2) rectangle (6,4);
\draw (0,2)--(0,4)--(0,6)--(2,6)--(2,4)--(4,4)--(6,4)--(6,2)--(6,0)--(4,0)--(2,0)--(2,2)--cycle;
\foreach \x in {(0,2),(0,4),(0,6),(2,6),(2,4),(4,4),(6,4),(6,0),(4,0),(2,0),(2,2)} \draw \x circle (1pt);
\draw(1,6) node[above] {\tiny 1};
\draw(3,4) node[above] {\tiny 2};
\draw(5,4) node[above] {\tiny 3};
\draw(1,2) node[below] {\tiny 3};
\draw(3,0) node[below] {\tiny 2};
\draw(5,0) node[below] {\tiny 1};
\draw(1,5) node {\tiny A};
\draw(3,3) node {\tiny B};
\draw(4,1) node {\tiny C};
\draw(-1.5,3) node {\small (O4)};
\end{tikzpicture}
\end{minipage}
\caption{In each of the four (up to symmetry) 3-cylinder diagrams  there are 3 choices of a free cylinder according to Proposition \ref{P:onefree}. We fix a label for each of the four 3-cylinder diagrams here, and a label for each cylinder in each such diagram.  The shading will be explained later.}
\label{F:O1-4}
\end{figure}
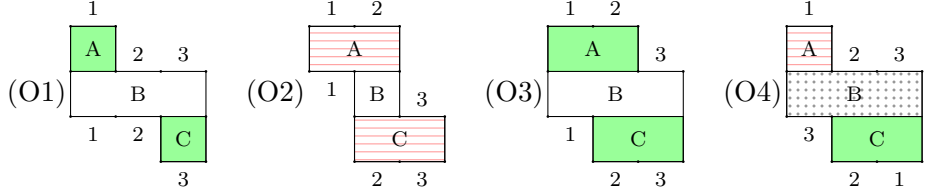
\begin{conv}
Let $N\in \{1,2,3,4\}$ and $L\in \{A,B,C\}$. We will say we are ``in case $ONL$" if $\cN$ contains a horizontally periodic translation surface whose cylinder diagram is of type $ON$ in Figure \ref{F:O1-4}, and cylinder $L$ is free. (As explained above, this necessarily means that the other two horizontal cylinders are $\cN$-parallel.) The ``$O$" in ``$ONL$" stands for ``odd"; we have chosen to include it to simplify any future attempts to unify our notation and analysis with that of \cite{NW} (an ``$H$" prefix could be used for the cases that occur for $\hyp$).
\end{conv}
We emphasize again that the cases are not mutually exclusive. Here $\cN$ is a fixed rank 2 affine invariant submanifold of $\odd$. We have already shown that for any such $\cN$, we must be in at least one of the cases above. Theorem \ref{T:main} will be established in the following sections by showing that some of the cases are impossible, and in the remaining cases $\cN=\prym$. 

\bold{Guide to the case analysis.} Lemma \ref{L:1} will rule out Cases O1A, O1C, O3A, O3C, O4C; Lemma \ref{L:2} will rule out Cases O2A, O2C, O4A; and Lemma \ref{L:O4B} will rule out Case O4B. Proposition \ref{P:cover} will give conditions under which, in Cases O1B, O2B and O3B, $\cN$ must be the Prym locus. Lemmas \ref{L:O1B} and \ref{L:O3B} verify these conditions for Cases O1B and O3B respectively, showing that $\cN$ is in fact the Prym locus. Lemma \ref{L:O2B} achieves the same conclusion by reduction to Case O2B.

Whenever an argument rules out several cases simultaneously, the reader may find it helpful to refer to pictures of one of the cases, and first trace through the argument in that particular case. We have chosen not to illustrate these arguments with specific cases because we wish to focus on the crucial common features in the various groups of cases.

\section{Impossible cases}\label{S:intro}

This section rules out all cases in Figure \ref{F:O1-4} which are shaded in any way (leaving only the ``white" cases O1B, O2B and O3B). 

\begin{lem}\label{L:1}
Let $\cN$ be an affine invariant submanifold, and let $M\in \cN$. Suppose that on $M$ there are horizontal cylinders $F$ and $L$, so that $F$ is free, and the bottom  boundary of $F$ is contained in the top  boundary of $L$. Suppose additionally that there is a horizontal saddle connection that is in the top  of $F$ and the bottom  of $L$. Then $L$ is free. 

The same statement holds with the roles of top and bottom reversed (so the top boundary of $F$ is contained in the bottom boundary of $L$, and there is a saddle connection in the bottom of $F$ and the top of $L$). 
\end{lem}

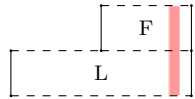
\begin{figure}[h!]
\centering
\begin{tikzpicture}[scale=0.30]
\draw [color=white, fill=red!40] (7,0) rectangle (7.5,4);
\draw (0,0)--(0,2);
\draw [style=dashed] (0,2) -- (4,2);
\draw (4,2)--(4,4);
\draw [style=dashed] (4,4) -- (8,4);
\draw (8,4)--(8,2)--(8,0);
\foreach \x in {(0,0),(0,2),(4,2),(4,4),(8,4),(8,2),(8,0)} \draw [fill=black] \x circle (1pt);
\draw [style=dashed] (0,0) -- (8,0);
\draw(4,1) node {\tiny L};
\draw(6,3) node {\tiny F};
\draw [style=dashed] (4,2)--(8,2);
\end{tikzpicture}
\caption{A special case of Lemma \ref{L:1}, see for instance $O1A$ or $O3A$ with $F=A, L=B$. After shearing $F$ and the equivalence class of $L$, there is a vertical cylinder $V$, shaded above, contained in the closure of $F$ union $L$. The dashed lines indicate that we do not have complete information on the boundaries of the horizontal cylinders. Lemma \ref{L:1} is somewhat more general than this picture, since the bottom of $F$ could be contained in the top of $L$ in a more complicated way.}
\label{F:L1}
\end{figure}

\begin{proof} 
After shearing $F$ and the cylinders in the equivalence class of $L$, there is a vertical cylinder $V$ that is contained in the closure of $F$ union $L$. (This uses the fact that there is a saddle connection $s$ in the top boundary of $F$ and the bottom boundary of $L$. First cylinders in the equivalence class of $L$ may be sheared so that some vertical trajectories which start at $s$ travel upwards through $L$ and into $F$. Then $F$ may be sheared so that these vertical trajectories hit $s$ and close up to form a cylinder.) 

Let $V'$ be a vertical cylinder that is $\cN$-parallel to $V$. Let $L_F$ denote the subset of $L$ of points whose upwards vertical trajectories enter $F$ before entering any other cylinder. (Thus $L_F$ may be thought of as the part of $L$ which lies directly below $F$.) Every time $V'$ goes through $F$, it must go down through $L_F$.  This gives the final equality in
\begin{eqnarray*}
P(V', F) &=& \frac{\Area(V'\cap F)}{\Area(V'\cap (M\setminus F))+\Area(V'\cap F)} \\
&\leq& \frac{\Area(V'\cap F)}{\Area(V'\cap L_F)+\Area(V'\cap F)}= P(V, F).
\end{eqnarray*}
(The final two quantities are equal  to the height of $F$ divided by the sum of the heights of $L$ and $F$.) Equality occurs if and only if $V'$ is contained in the closure of $L_F$ union $F$. By Proposition \ref{P:P} (with $\cE=\{F\}$, $X=V, Y=V'$), the proportion of $V'$ in $F$ is equal to the proportion of $V$ in $F$. Hence $V'$ is contained in the closure of $L_F$ union $F$. 

Thus the equivalence class $\cV$ of all cylinders $\cN$-parallel to $V$ is contained in the closure of $F$ union $L$. By Proposition \ref{P:P} (with $\cE=\cV$, $X=L$, and $Y$ any cylinder $\cN$-parallel to $L$), the proportion of $L$ in $\cV$ is equal to the proportion of any cylinder $\cN$-parallel to $L$ in $\cV$.  Hence, we see that $L$ is not $\cN$-parallel to any horizontal cylinder except itself and possibly $F$. Since $F$ is free, $L$ cannot be $\cN$-parallel to $F$. We conclude that $L$ is free. 
\end{proof}

\begin{cor}
In the analysis of rank 2 affine invariant submanifolds of $\odd$, the following cases are impossible: O1A, O1C, O3A, O3C, O4C. 
\end{cor}

These cases are shaded solid in Figure \ref{F:O1-4}. 

\begin{proof}
In all five cases, set $L=B$. In each case, let $F$ be the choice of free cylinder, for example $F=A$ in case O1A. Then in each case Lemma \ref{L:1} gives that $L$ is free. However in each of these cases $F$ is the only free horizontal cylinder. 
\end{proof}

\begin{lem}\label{L:2}
Let $\cN$ be an affine invariant submanifold, and suppose $M\in \cN$ has two horizontal cylinders $K$ and $L$ such that the closure of the union of $K$ and $L$ is not the whole surface. Suppose that there is a saddle connection $s$ which is both on the top and bottom of $L$. Suppose that the top boundary of $L$ is equal to $s$ union the bottom boundary of $K$. Then $\{K,L\}$ is not an equivalence class of $\cN$--parallel cylinders. 

The same statement holds with the roles of top and bottom reversed. 
\end{lem}

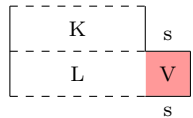
\begin{figure}[h!]
\centering
\begin{tikzpicture}[scale=0.30]
\draw [color=white, fill=red!40] (6,0) rectangle (8, 2);
\draw (0,0)--(0,4);
\draw (6,4)--(6,2);
\draw (8,2)--(8,0); 
\draw (6,2)--(8,2); 
\draw (6,0)--(8,0);
\draw [style=dashed] (0,4)--(6,4);
\draw [style=dashed] (0,2)--(6,2);
\draw [style=dashed] (0,0)--(6,0);
\draw(7,2) node[above] {\tiny s};
\draw(7,0) node[below] {\tiny s};
\draw(7,1) node {\tiny V};
\draw(3,3) node {\tiny K};
\draw(3,1) node {\tiny L};
\end{tikzpicture}
\caption{A special case of Lemma \ref{L:2}, see for instance case $O2$ with $K=B$ and $L=C$, or case $O4$ with $K=C$ and $L=B$.}
\end{figure}

\begin{proof}
Suppose by contradiction that $\{K,L\}$ is an equivalence class of $\cN$--parallel cylinders.

After shearing the whole surface, there is a vertical cylinder $V$ that contains $s$. Let $\cV$ be the equivalence class of cylinders $\cN$-parallel to $V$. By Proposition \ref{P:P} (with $\cE=\{K, L\}$, $X=V$, and $Y\in \cV$), the proportion of $V$ in $\cE$ must be equal to the proportion of any cylinder $Y$ that is $\cN$-parallel to $V$. Since we have $P(V,\cE)=1$, it follows $P(Y,\cE)=1$. Hence, each cylinder in $\cV$ is contained in the closure of $K$ union $L$. 

Since $K$ union $L$ is not the whole surface, we see that $\cV$ does not cover $K$ union $L$.  It follows by an argument similar to that in the proof of the previous lemma that $P(K,\cV)<P(L, \cV)$. By Proposition \ref{P:P} this contradicts the assumption that $K$ and $L$ are $\cN$-parallel. 
\end{proof}

\begin{cor}
In the analysis of rank 2 affine invariant submanifolds of $\odd$, the following cases are impossible: O2A, O2C, O4A. 
\end{cor}

These cases are shaded with horizontal lines in Figure \ref{F:O1-4}. 

\begin{proof}
In case O2A, set $K=B, L=C$. In case O2C, set $K=B, L=A$. In case O4A, set $K=C, L=B$. The assumption of these cases is that $\{K,L\}$ is an equivalence class of $\cN$-parallel cylinders, but Lemma \ref{L:2} shows that this is impossible. 
\end{proof}

\begin{lem}\label{L:O4B}
In the analysis of rank 2 affine invariant submanifolds of $\odd$, case O4B is impossible. 
\end{lem}

\begin{figure}[h!]
\centering
\begin{tikzpicture}[scale=0.5]
\draw [color=white, pattern=vertical lines, pattern color=red!40] (2,0) rectangle (4,4);
\draw [color=white, fill=blue!40] (0,2)--(2,2)--(0,4)--cycle;
\draw [color=white, fill=blue!40] (6,4)--(6,2)--(4,4)--cycle;
\draw (0,2)--(0,4)--(-1,6)--(1,6)--(2,4)--(4,4)--(6,4)--(6,2)--(6,0)--(4,0)--(2,0)--(2,2)--cycle;
\foreach \x in {(0,2),(0,4),(-1,6),(1,6),(2,4),(4,4),(6,4),(6,2),(6,0),(4,0),(2,0),(2,2)} \filldraw[fill=black] \x circle (2pt);
\draw(0,6) node[above] {\tiny 1};
\draw(3,4) node[above] {\tiny 2};
\draw(5,4) node[above] {\tiny 3};
\draw(1,2) node[below] {\tiny 3};
\draw(3,0) node[below] {\tiny 2};
\draw(5,0) node[below] {\tiny 1};
\draw(-1.5,3) node {\small (O4)};
\draw(3,2) node {\tiny V};
\draw (0.5,2.5) node {\tiny K} (5.5,3.5) node {\tiny K};
\end{tikzpicture}
\caption{In the proof of Lemma \ref{L:O4B}, eventually there is a vertical cylinder $V$ shaded with vertical lines above, and a cylinder $K$ shaded solid  above.}
\label{F:O4B}
\end{figure}
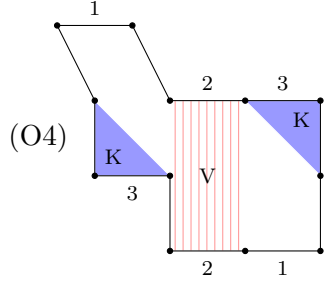

\begin{proof}
After shearing $B$ and $\{A,C\}$ there is a vertical cylinder $V$ which contains the saddle connection (2). There is also a cylinder $K$ which is contained in $B$, and contains the saddle connection (3).  See Figure \ref{F:O4B}. 

By Proposition \ref{P:P} (with $\cE=\{B\}, X=K$ and $Y$ any cylinder $\cN$-parallel to $K$), the proportion of $K$ in $\cE$, which is one, and the proportion of any cylinder $\cN$-parallel to $K$, which cannot be one, we see that $K$ is free. By deforming $K$ (stretching it horizontally using a product of cylinder shear and cylinder stretch), we may assume that the length of (3) is equal to the length of (1).

Since any vertical cylinder other than $V$ must intersect $K$, Proposition \ref{P:P} (with $\cE=\{K\}, X=V$ and $Y$ any cylinder $\cN$-parallel to $V$) implies that $V$ is free because the proportion of $V$ in $\{K\}$ is zero, while the proportion of any cylinder $\cN$-parallel to $V$ cannot be zero. Then Proposition \ref{P:P} (with $\cE=\{V\}, X=C, Y=A$), implies that the proportion of $A$ in $\{V\}$, which is zero, equals the proportion of $C$ in $\{V\}$, which is non-zero, and contradicts the assumption that $A$ and $C$ are $\cN$-parallel.
\end{proof}

\section{The Prym locus}\label{S:intro}

This section describes the conditions under which, in the remaining cases, we may assume that $\cN$ is the Prym locus $\prym$.  


\begin{prop}\label{P:cover}
In cases O1B, O2B and O3B, suppose that the cylinders $A$ and $C$ are affinely equivalent via an affine map with derivative -1, and assume this map sends the saddle connection (1) to the saddle connection (3). Then $\cN=\prym$. 
\end{prop}

In the assumption of the proposition, the involution is not required a priori to extend to the whole surface. The proposition will show this, and this will imply that the surface is in the Prym locus.

In terms of our pictures, the assumption says that the top rectangle can be rotated by $\pi$ and placed exactly on top of the bottom cylinder, so that saddle connection (1) ends up exactly on top of (3). 

\begin{proof}
Let $M\in \cN$ be the translation surface of the specified type, depending on the case. Under the given condition, we see that $M$ admits an involution  which fixes $B$ and exchanges $A$ and $C$. This involution has four fixed points as indicated in Figure~\ref{fig:prym-inv}, and $M$ can be  easily seen to lie in the Prym locus $\prym$. 

By stretching the middle cylinder $B$ (which is free), we may assume that the ratio of the modulus of $A$ to the modulus of $B$ is irrational. Hence the Veech Dichotomy \cite{V} gives that $M$ does not lie on a closed orbit. Thus the orbit closure of $M$ is a 4 (complex) dimensional affine invariant submanifold, which must be contained in both $\cN$ and $\prym$. Since both $\cN$ and $\prym$ are 4 dimensional, we get that $\cN=\prym$.
\begin{figure}[htb]
\begin{minipage}[t]{0.32\linewidth}
\centering
\begin{tikzpicture}[scale=0.33]
\draw (0,2)--(0,4)--(0,6)--(2,6)--(2,4)--(4,4)--(6,4)--(6,2)--(6,0)--(4,0)--(4,2)--(2,2)--cycle;
\foreach \x in {(0,2),(0,4),(0,6),(2,6),(2,4),(4,4),(6,4),(6,2),(6,0),(4,0),(4,2),(2,2)} \filldraw[fill=black] \x circle (3pt);
\draw(1,6) node[above] {\tiny 1};
\draw(3,4) node[above] {\tiny 2};
\draw(5,4) node[above] {\tiny 3};
\draw(1,2) node[below] {\tiny 1};
\draw(3,2) node[below] {\tiny 2};
\draw(5,0) node[below] {\tiny 3};
\draw(-1.5,3) node {\small (O1)};
\foreach \x in {(3,3)} \draw [fill=black] \x circle (4pt);
\foreach \x in {(3,2),(3,4))} \draw [fill=red!40] \x circle (4pt);
\foreach \x in {(0,3),(6,3)} \draw [fill=green!40] \x circle (4pt);
\draw [->] (3cm+10pt,3) arc (0:300:10pt);
\end{tikzpicture}
\end{minipage}
\begin{minipage}[b]{0.32\linewidth}
\centering
\begin{tikzpicture}[scale=0.33]
\draw (0,4)--(0,6)--(2,6)--(4,6)--(4,4)--(4,2)--(6,2)--(6,0)--(4,0)--(2,0)--(2,2)--(2,4)--cycle;
\foreach \x in {(0,4),(0,6),(2,6),(4,6),(4,4),(4,2),(6,2),(6,0),(4,0),(2,0),(2,2),(2,4)} \filldraw[fill=black] \x circle (3pt);
\draw(1,6) node[above] {\tiny 1};
\draw(3,6) node[above] {\tiny 2};
\draw(5,2) node[above] {\tiny 3};
\draw(1,4) node[below] {\tiny 1};
\draw(3,0) node[below] {\tiny 2};
\draw(5,0) node[below] {\tiny 3};
\draw(-1.5,3) node {\small (O2)};
\foreach \x in {(3,3)} \draw [fill=black] \x circle (4pt);
\foreach \x in {(3,0),(3,6))} \draw [fill=red!40] \x circle (4pt);
\foreach \x in {(2,3),(4,3)} \draw [fill=green!40] \x circle (4pt);
\draw [->] (3cm+10pt,3) arc (0:300:10pt);
\end{tikzpicture}
\end{minipage}
\begin{minipage}[b]{0.32\linewidth}
\centering
\begin{tikzpicture}[scale=0.33]
\draw (0,2)--(0,4)--(0,6)--(2,6)--(4,6)--(4,4)--(6,4)--(6,2)--(6,0)--(4,0)--(2,0)--(2,2)--cycle;
\foreach \x in {(0,2),(0,4),(0,6),(2,6),(4,6),(4,4),(6,4),(6,0),(4,0),(2,0),(2,2)} \filldraw[fill=black] \x circle (3pt);
\draw(1,6) node[above] {\tiny 1};
\draw(3,6) node[above] {\tiny 2};
\draw(5,4) node[above] {\tiny 3};
\draw(1,2) node[below] {\tiny 1};
\draw(3,0) node[below] {\tiny 2};
\draw(5,0) node[below] {\tiny 3};
\draw(-1.5,3) node {\small (O3)};
\foreach \x in {(3,3)} \draw [fill=black] \x circle (4pt);
\foreach \x in {(3,0),(3,6))} \draw [fill=red!40] \x circle (5pt);
\foreach \x in {(0,3),(6,3)} \draw [fill=green!40] \x circle (5pt);
\draw [->] (3cm+10pt,3) arc (0:300:10pt);
\end{tikzpicture}
\end{minipage}
\caption{Under appropriate assumptions on the top and bottom cylinders, the surfaces in the cases above admit involutions with four fixed points. These fixed points are the zero, together with the large dots drawn above. In these pictures, the involution is given by rotation by $\pi$ about the central point.}
\label{fig:prym-inv}
\end{figure}
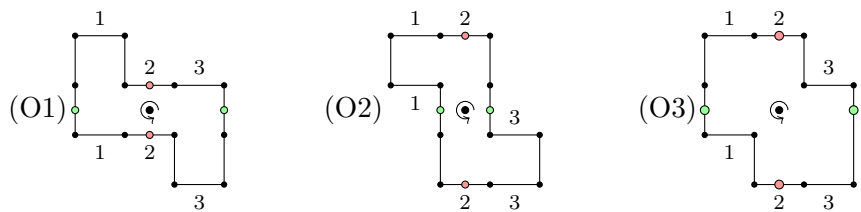

\end{proof}

\section{Cases O1B and O2B, and two slit tori}\label{S:intro}

In this section we show that cases O1B and O2B imply  $\cN=\prym$.  

\begin{lem}\label{L:P12}
Let $P_1$ and $P_2$ be two flat slit tori, as shown in Figure \ref{F:P12}: The corresponding vertical edges in both pictures are assumed to have the same length. Suppose that $\Area(P_1) \geq \Area(P_2)$, and for every cylinder $L_1$ on $P_1$ there is a cylinder $L_2$ on $P_2$ of the same slope, with 
$$\frac{\Area(L_1)}{\Area(P_1)}= \frac{\Area(L_2)}{\Area(P_2)}.$$
Then $P_2$ is isometric to $P_1$. 
\end{lem}

\begin{figure}[htb]
\begin{minipage}[t]{0.45\linewidth}
\centering
\begin{tikzpicture}[scale=0.2]
\draw (0,0)--(0,2)--(0,8)--(12,8)--(12,2)--(12,0)--cycle;
\foreach \x in {(0,0),(0,2),(0,8),(12,8),(12,2),(12,0)} \draw [fill=black] \x circle (5pt);
\draw(0,5) node[left] {\tiny $a$};
\draw(12,5) node[right] {\tiny $a$};
\draw(6,8) node[above] {\tiny $b$};
\draw(6,0) node[below] {\tiny $b$};
\draw(-3.5,4) node {\small ($P_1$)};
\end{tikzpicture}
\end{minipage}
\begin{minipage}[t]{0.45\linewidth}
\centering
\begin{tikzpicture}[scale=0.2]
\draw (0,0)--(0,2)--(0,8)--(10,8)--(10,2)--(10,0)--cycle;
\foreach \x in {(0,0),(0,2),(6,8),(10,2),(10,0)} \draw [fill=black] \x circle (5pt);
\foreach \x in {(0,8),(10,8),(4,0)} \draw [fill=black] \x circle (2.5pt);
\draw(0,5) node[left] {\tiny $a$};
\draw(10,5) node[right] {\tiny $a$};
\draw(3,8) node[above] {\tiny $b_1$};
\draw(7,0) node[below] {\tiny $b_1$};
\draw(8,8) node[above] {\tiny $b_2$};
\draw(2,0) node[below] {\tiny $b_2$};
\draw(-3.5,4) node {\small ($P_2$)};
\end{tikzpicture}
\end{minipage}
\caption{Statement of Lemma \ref{L:P12}: two slit tori. The identifications within each are given by the letters. The bottom (unlabeled) vertical saddle connections are  not identified with anything; they are the slits. The slits are vertical, and of the same length on both $P_1$ and $P_2$. The saddle connection $a$ has the same length on both $P_1$ and $P_2$. On $P_2$, it may be that one of $b_1$ and $b_2$ has zero length. On $P_2$, the individual segments $b_1$ and $b_2$ are not saddle connections, although their union is. In $P_1$ the upper corners of the rectangle correspond to the endpoint of a slit. In $P_2$ the upper corners correspond to a non-singular point. This is indicated by making the dots slightly smaller.} 
\label{F:P12}
\end{figure}
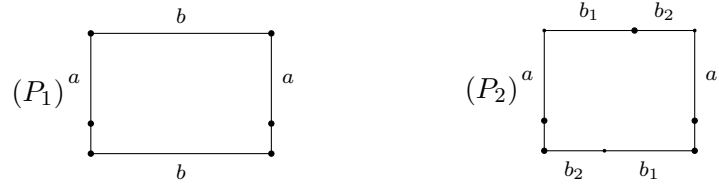

\begin{proof}
Let $T_i=\bC/\Lambda_i$ be the torus obtained by filling in the slit on $P_i$, so that $T_i$ is equal to $P_i$ with a vertical slit. Let $r=\frac{\Area(P_2)}{\Area(P_1)}$. 
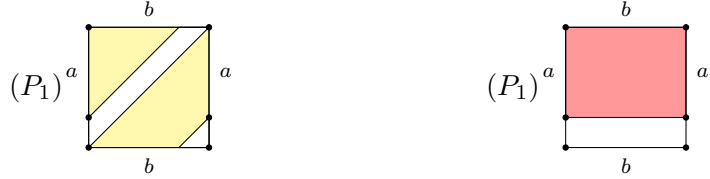
\begin{figure}
\begin{minipage}[t]{0.49\linewidth}
\centering
\begin{tikzpicture}[scale=0.2]
\draw [fill=yellow!40] (0,0) rectangle (8,8);
\draw [fill=white] (0,0) -- (8,8) -- (6,8) -- (0,2) --cycle;
\draw [fill=white] (8,0) -- (8,2) -- (6,0) --cycle;
\draw (0,0)--(0,2)--(0,8)--(8,8)--(8,2)--(8,0)--cycle;
\foreach \x in {(0,0),(0,2),(0,8),(8,8),(8,2),(8,0)} \draw [fill=black] \x circle (5pt);
\draw(0,5) node[left] {\tiny $a$};
\draw(8,5) node[right] {\tiny $a$};
\draw(4,8) node[above] {\tiny $b$};
\draw(4,0) node[below] {\tiny $b$};
\draw(-3.5,4) node {\small ($P_1$)};
\end{tikzpicture}
\end{minipage}
\begin{minipage}[t]{0.49\linewidth}
\centering
\begin{tikzpicture}[scale=0.2]
\draw [fill=red!40] (0,2) rectangle (8,8);
\draw (0,0)--(0,2)--(0,8)--(8,8)--(8,2)--(8,0)--cycle;
\foreach \x in {(0,0),(0,2),(0,8),(8,8),(8,2),(8,0)} \draw [fill=black] \x circle (5pt);
\draw(0,5) node[left] {\tiny $a$};
\draw(8,5) node[right] {\tiny $a$};
\draw(4,8) node[above] {\tiny $b$};
\draw(4,0) node[below] {\tiny $b$};
\draw(-3.5,4) node {\small ($P_1$)};
\end{tikzpicture}
\end{minipage}
\caption{Proof of Lemma \ref{L:P12}. On the left, a cylinder $L_1$ on $P_1$ in the direction of the line from the bottom left to the top right corner of the rectangle. On the right, a horizontal cylinder $H_1$ on $P_1$.} 
\label{F:LH}
\end{figure}
Let $L_1$ and $H_1$ be the cylinders shown in Figure \ref{F:LH}. Let $\alpha, \beta\in \bC$ be the holonomies of the core curves of $H_1$ and $L_1$, respectively. (Any   orientation for these core curves can be fixed.) Note that $\Lambda_1=\alpha\bZ\oplus\beta \bZ$. 

Since the lengths of the vertical edges are the same for both $P_1$ and $P_2$, we see that $r\alpha$ is the holonomy of the horizontal cylinder on $P_2$. 

By assumption, there is a cylinder $L_2$ on $P_2$ in the same direction as $L_1$. In this direction $P_2$ is the cylinder $L_2$ union a parallelogram from the slit to itself, as shown in Figure \ref{F:L2}.
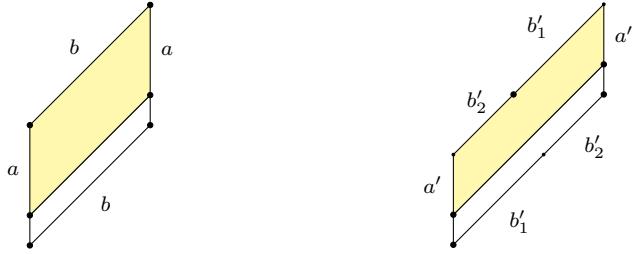
\begin{figure}
\begin{minipage}[c]{0.4\linewidth}
\centering
\begin{tikzpicture}[scale=0.2]
\draw (0,0) -- (8,8) -- (8, 10) -- (0,2) -- cycle; 
\draw [fill=yellow!40] (0,2) -- (8,10) -- (8, 16) -- (0,8) -- cycle; 
\foreach \x in {(0,0),(8,8),(8,10),(0,2), (8, 16), (0,8) } \draw [fill=black] \x circle (5pt);
\draw(0,5) node[left] {\tiny $a$};
\draw(8,13) node[right] {\tiny $a$};
\draw(4,4) node[below right] {\tiny $b$};
\draw(4,12) node[above left] {\tiny $b$};
\end{tikzpicture}
\end{minipage}
\begin{minipage}[c]{0.5\linewidth}
\centering
\begin{tikzpicture}[scale=0.2]
\draw (0,0) -- (10,10) -- (10, 12) -- (0,2) -- cycle; 
\draw [fill=yellow!40] (0,2) -- (10,12) -- (10, 16) -- (0,6) -- cycle; 
\foreach \x in {(0,2),(10,12), (4,10), (0,0), (10,10)} \draw [fill=black] \x circle (5pt);
\foreach \x in {(0,6),(10,16), (6,6)} \draw [fill=black] \x circle (2.5pt);
\draw(0,4) node[left] {\tiny $a'$};
\draw(10,14) node[right] {\tiny $a'$};
\draw(3,8) node[above left] {\tiny $b'_2$};
\draw(7,13) node[above left] {\tiny $b'_1$};
\draw(3, 3) node[below right] {\tiny $b'_1$};
\draw(8, 8) node[below right] {\tiny $b'_2$};
\end{tikzpicture}
\end{minipage}
\caption{Proof of Lemma \ref{L:P12}. On the left, $P_1$ in the direction of $L_1$, shaded solid. On the right, $P_2$ in the direction of $L_2$, shaded solid. $L_1$ and $L_2$ are in the same direction. The slits are the same length on both pictures, and it is assumed that the shaded solid region occupies an equal fraction of each picture.}
\label{F:L2}
\end{figure}
In Figure \ref{F:L2}, the height of the white strips is the same on $P_1$ and $P_2$, and by assumption the shaded solid regions (and hence their complements) occupy the same fraction of the area on both $P_1$ and $P_2$. Elementary geometry gives that the length of the top boundary of the parallelogram defining $P_2$ in Figure \ref{F:L2} is $r$ times the corresponding quantity for $P_1$.  Thus the core curve of $L_2$ is $r\beta$. 

Thus $\Lambda_2 \supseteq r(\alpha\bZ\oplus\beta \bZ) = r\Lambda_1$. In particular, we have ${\rm covol}(r\Lambda_1)= r^2{\rm covol}(\Lambda_1) \geq {\rm covol}(\Lambda_2)$. Since $r$ is the ratio of the covolumes of $\Lambda_2$ and $\Lambda_1$, and $r \leq 1$ by assumption, this gives $r=1$ and $\Lambda_2=\Lambda_1$. Therefore $P_1$ and $P_2$ are isometric as desired. 
\end{proof}


\begin{lem}\label{L:O1B}
In Case O1B, $\cN=\prym$. 
\end{lem}

\begin{proof}
\begin{figure}[htb]
\centering
\begin{tikzpicture}[scale=0.60]
\draw (0,2)--(0,4)--(0,6)--(2,6)--(2,4)--(4,4)--(6,4)--(6,2)--(7,0)--(5,0)--(4,2)--(2,2)--cycle;
\draw [fill=orange!40] (2,2) rectangle (4,4);
\draw [color=white] (0,2)--(0,4);
\draw [style=densely dotted] (0,2)--(0,4);
\draw [color=white] (2,2)--(2,4);
\draw [style=densely dotted] (2,2)--(2,4);
\draw [color=white] (4,2)--(4,4);
\draw [style=densely dotted] (4,2)--(4,4);
\draw [color=white] (6,2)--(6,4);
\draw [style=densely dotted] (6,2)--(6,4);
\foreach \x in {(0,2),(0,4),(0,6),(2,6),(2,4),(4,4),(6,4),(6,2),(7,0),(5,0),(4,2),(2,2)} \draw \x circle (1pt);
\draw(1,6) node[above] {\tiny 1};
\draw(3,4) node[above] {\tiny 2};
\draw(5,4) node[above] {\tiny 3};
\draw(1,2) node[below] {\tiny 1};
\draw(3,2) node[below] {\tiny 2};
\draw(6,0) node[below] {\tiny 3};
\draw(-1.5,3) node {\small (O1)};
\draw(3,3) node {\tiny V};
\draw(1,4) node {\tiny $P_1$};
\draw(5.1,2) node {\tiny $P_2$};
\end{tikzpicture}
\caption{Proof of Lemma \ref{L:O1B}.}
\label{F:V}
\end{figure}
By Proposition \ref{P:cover}, it suffices to show that cylinders $A$ and $C$ are isometric. Shearing $B$ we may assume that there is a vertical cylinder $V$ which is contained in the closure of $B$ as shown in Figure \ref{F:V}. Shearing the complement of $B$ we may assume that there is a vertical cylinder through (1). Let $M$ be the resulting surface, shown in Figure \ref{F:V}. 

There are three vertical saddle connections homologous to the core curve of $V$; these are drawn with dotted lines in Figure \ref{F:V}. Cutting along these three saddle connections gives one vertical cylinder $V$, and two slit tori, which we denote by $P_1$ and $P_2$. Without loss of generality we can assume that $\Area(P_1) \geq \Area(P_2)$. 

We now claim that $P_1$ and $P_2$ satisfy the assumptions of Lemma \ref{L:P12}. Indeed, suppose that there is a cylinder $L_1 \subset P_1$. We also consider $L_1$ as a cylinder on $M$. By Proposition \ref{P:P} (with $\cE$ the equivalence class of cylinders $\cN$-parallel to $L_1$, $X=A, Y=C$), the proportion of $A$ in the equivalence class of cylinders $\cN$-parallel to $L_1$, which is non-zero, must be equal to the proportion of $C$ in the equivalence class of cylinders $\cN$-parallel to $L_1$.  Since the proportion of $C$ in the cylinder $L_1$ is zero, there must be some cylinder $L_2$ which is $\cN$-parallel to $L_1$ which crosses $C$.

By Proposition \ref{P:P} (with $\cE=\{B\}$, $X=V$), the proportion of $V$ in $\cE$ is one, and since there are no cylinders parallel to $V$ which are entirely contained in $B$, we see that $V$ is free. Proposition \ref{P:P} (with $\cE=\{V\}, X=L_1, Y=L_2$) also implies that $L_2$ does not intersect $V$ because the proportion of $L_1$ in $V$ is zero and $L_2$ is $\cN$-parallel to $L_1$. Hence, $L_2$ does not intersect any of the saddle connections homologous to the core curve of $V$. Thus, $L_2$ is contained entirely in $P_2$.

Consider the vertical cylinder $V_A$ which contains $A$. By Proposition~\ref{P:P} (with $\cE$ equal to the equivalence class of $V_A$, $X=A, Y=C$),  the proportion of $A$ in the equivalence class of cylinders parallel to $V_A$, which is one, must be equal to the proportion of $C$ in that equivalence class.  Hence, there must be vertical cylinder $V_C$ which contains $C$. The closure of $V_C$ is equal to the closure of $P_2$. By Proposition \ref{P:P} (with $\cE=\{B\}, X=V_A, Y=V_C$), the equation that the proportion of $V_A$ in $B$ is equal to the proportion of $V_C$ in $B$ can be simplified with elementary algebra to prove that the height of the $C$ is equal to the height of $A$.

By Proposition \ref{P:P} (with $\cE=\{L_1,L_2\}, X=V_A, Y=V_C$), the proportion of $V_A$ in $L1\cup L_2$ is equal to the proportion of $V_C$ in $L_1\cup L_2$, or equivalently
$$\frac{\Area(L_1)}{\Area(P_1)}= \frac{\Area(L_2)}{\Area(P_2)}.$$

Thus the claim is proved, and Lemma \ref{L:P12} gives that $P_1$ and $P_2$ are isometric. Proposition~\ref{P:cover} now gives that $\cN=\prym$. 
\end{proof}

\begin{lem}\label{L:O2B}
In Case O2B, $\cN=\prym$. 
\end{lem}


\begin{proof}
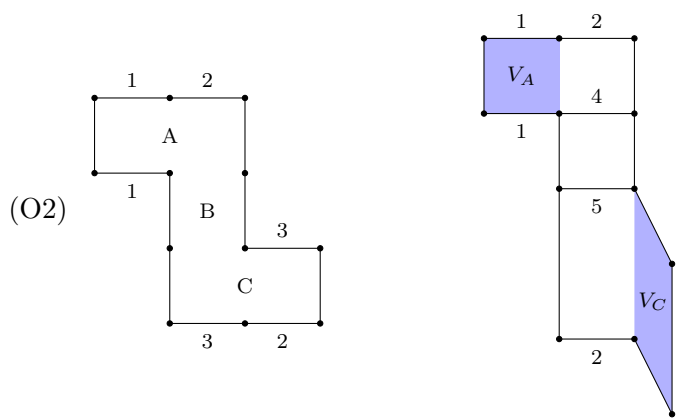
\begin{figure}[htb]
\centering
\begin{minipage}[c]{0.4\linewidth}
\begin{tikzpicture}[scale=0.5]
\draw (0,4)--(0,6)--(2,6)--(4,6)--(4,4)--(4,2)--(6,2)--(6,0)--(4,0)--(2,0)--(2,2)--(2,4)--cycle;
\foreach \x in {(0,4),(0,6),(2,6),(4,6),(4,4),(4,2),(6,2),(6,0),(4,0),(2,0),(2,2),(2,4)} \filldraw[fill=black] \x circle (2pt);
\draw(1,6) node[above] {\tiny 1};
\draw(3,6) node[above] {\tiny 2};
\draw(5,2) node[above] {\tiny 3};
\draw(1,4) node[below] {\tiny 1};
\draw(3,0) node[below] {\tiny 3};
\draw(5,0) node[below] {\tiny 2};
\draw(2,5) node {\tiny A};
\draw(3,3) node {\tiny B};
\draw(4,1) node {\tiny C};
\draw(-1.5,3) node {\small (O2)};
\end{tikzpicture}
\end{minipage}
\begin{minipage}[c]{0.4\linewidth}
\centering
\begin{tikzpicture}[scale=0.5]
\fill[blue!30] (0,8) -- (0,6) -- (2,6) -- (2,8) -- cycle;
\fill[blue!30] (4,4) -- (4,0) -- (5,-2) -- (5,2) -- cycle;


\draw (2,0) -- (4,0) -- (5,-2) -- (5,2) -- (4,4) -- (4,8) -- (0,8) -- (0,6) -- (2,6) -- cycle;
\draw (2,6) -- (4,6) (2,4) -- (4,4);
\foreach \x in {(0,8), (0,6), (2,8), (2,6), (2,4), (2,0), (4,8), (4,6), (4,4), (4,0), (5,2), (5,-2)} \filldraw[fill=black] \x circle (2pt);
\draw (1,8) node[above] {\tiny $1$};
\draw (1,6) node[below] {\tiny $1$};
\draw (3,8) node[above] {\tiny $2$};
\draw (3,0) node[below] {\tiny $2$};
\draw (3,6) node[above] {\tiny $4$};
\draw (3,4) node[below] {\tiny $5$};
\draw (1,7) node {\tiny $V_A$};
\draw (4.5,1) node {\tiny $V_C$};
\end{tikzpicture}
\end{minipage}
\caption{On the left, we have reproduced the cylinder diagram for Case O2. On the right, we have illustrated the switching of cylinder decompositions into Case O1B in the vertical direction.}
\label{F:O2B}
\end{figure}

We start by shearing $A$ (and $C$) so that there is a vertical cylinder $V_A$ that contains the saddle connection $(1)$ and is contained in the closure of $A$.  Since $B$ is free, we can shear it so that it contains a vertical saddle connection  (see Figure~\ref{F:O2B}). Since $A$ and $C$ are $\cN$-parallel, by Proposition~\ref{P:P} (with $\cE$ the equivalence class of $V_A$, and $X=A, Y=C$), the proportion of $A$ in the equivalence class of cylinders $\cN$-parallel to $V_A$, which is non-zero, is equal to the proportion of $C$ in $\cE$, which implies that there exists a vertical cylinder $V_C$ in the equivalence class of $V_A$ that intersects $C$. Also by Proposition~\ref{P:P} (with $\cE=\{B\}, X=V_A, Y=V_C$), the proportion of $V_A$ in $\{B\}$, which is zero, must be equal to the proportion of $V_C$ in $\{B\}$.  Hence, $V_C$ cannot intersect $B$, since $V_A$ does not. Therefore, $V_C$ does not cross any of the saddle connections $(2),(4),(5)$. We can then conclude that 
$V_C$ is entirely contained in the closure of $C$. But $\overline{C}$ is a slit torus, hence it is the union of $V_C$ and a rectangle bounded by $(2),(5)$ and the boundaries of $V_C$. Thus the surface admits a cylinder decomposition in the vertical direction in model $O1$. Moreover, since $V_A$ and $V_C$ are $\cN$-parallel, we are in Case $O1B$. We can now use Lemma~\ref{L:O1B} to conclude that $\cN=\prym$.
\end{proof}

\section{Case O3B}\label{S:intro}

In this section we show   

\begin{lem}\label{L:O3B}
Case O3B implies that $\cN=\prym$.
\end{lem}

\begin{proof}
We will again show that $A$ and $C$ are isometric. 

We begin by shearing cylinders $A$ and $B$ so that saddle connection $(1)$ lies directly above itself. Then there is a vertical cylinder $V_1$ passing through $A$ and $B$, but not $C$ (see Figure \ref{F:O3BDavid}).  Call this Model I. By Proposition \ref{P:P} (with $\cE$ the equivalence class of $V_1$ and $X=A, Y=C$), the proportion of $A$ in the equivalence class of cylinders $\cN$-parallel to $V_1$, which is nonzero, is equal to the proportion of $C$ in $\cE$.  Since $C$ does not intersect the cylinder $V_1$, there is a vertical cylinder $V_2$ that is $\cN$-parallel to $V_1$ that intersects $C$.

Every trajectory ascending vertically from $C$ must pass through $B$, and so $V_2$ passes through each of $B$ and $C$ an equal number of times, say $m > 0$ times. Say that $V_2$ passes through $A$ $n \geq 0$ times.  Note that $m \geq n$. 

Let $\cE = \{A, C\}$.  Let $h(\cdot)$ denote the height of a cylinder, and let $\ell(\cdot)$ denote the length of a saddle connection. By Proposition \ref{P:P}, $P(V_1, \{A,C\}) = P(V_2, \{A,C\})$.  This yields
$$ \frac{h(A)}{h(A)+h(B)}  =  \frac{n h(A) + m h(C)}{n h(A) + m (h(B) + h(C))},$$
which holds if and only if 
$$\frac{h(C)}{h(A)}   =  \frac{m-n}{m}.$$

\noindent Hence, $h(C) \leq h(A)$.

\begin{figure}
\centering
\begin{tikzpicture}[scale=0.5]
\draw (-7,0) -- (-9,0) -- (-9,4) -- (-4,4) -- (-4,2) -- (-1,2) -- (-1,0) -- (1,-2) -- (-5,-2) -- cycle;
\draw (-7,0) -- (-7,4);
\foreach \x in {(-7,0), (-9,0), (-9,2), (-9,4), (-7,4), (-4,4), (-4,2), (-1,2), (-1,0), (1,-2), (-2,-2), (-5,-2)} \filldraw[fill=black] \x circle (2pt);
\draw (-8,4) node[above] {\tiny $1$}; 
\draw (-8,0) node[below] {\tiny $1$};
\draw (-5.5,4) node[above] {\tiny $2$};
\draw (-3.5,-2) node[below] {\tiny $2$};
\draw (-2.5,2) node[above] {\tiny $3$};
\draw (-.5,-2) node[below] {\tiny $3$};
\draw (-8,2) node {\tiny $V_1$};

\draw (5,0) -- (3,0) -- (3,2) -- (1,4) -- (6,4) -- (8,2) -- (11,2) -- (11,-2) -- (5,-2) -- cycle;
\draw (8,2) -- (8,-2);
\foreach \x in {(5,0), (1,4), (6,4), (8,2), (11,-2), (5,-2), (11,0), (3,0), (3,2), (3,4), (8,2), (11,2), (8,-2)} \filldraw[fill=black] \x circle (2pt);
\draw (2,4) node[above] {\tiny $1$}; 
\draw (4,0) node[below] {\tiny $1$};
\draw (4.5,4) node[above] {\tiny $2$};
\draw (6.5,-2) node[below] {\tiny $2$};
\draw (9.5,2) node[above] {\tiny $3$};
\draw (9.5,-2) node[below] {\tiny $3$};
\end{tikzpicture}
\caption{Proof of Lemma \ref{L:O3B}: Model I is pictured on the left and Model II on the right.}
\label{F:O3BDavid}
\end{figure}
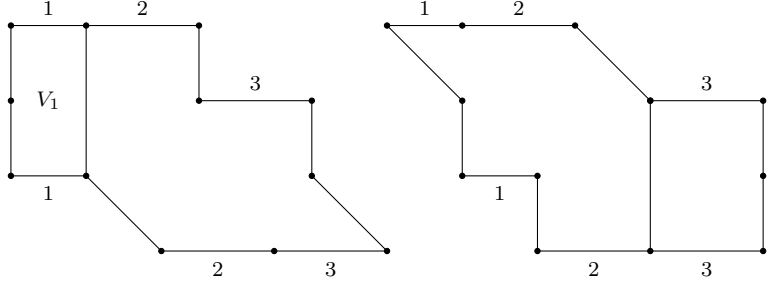

On the other hand, if we shear $B$ and $C$ so that $(3)$ lies directly above itself, and if we repeat the same argument above, then we get $h(A) \leq h(C)$. Call this Model II (depicted in Figure \ref{F:O3BDavid}). Hence, $h(A) = h(C)$. Moreover, $n = 0$, which means that $V_2$ is contained in the closure of $B \cup C$.

Next consider the arrangement of Model I.  Recall that we proved there is another cylinder $V_2$ that is $\cN$-parallel to $V_1$.  If we use Smillie-Weiss \cite{SW2} to get a third cylinder in the vertical direction, then we are done unless the resulting surface satisfies Case O3B.  In this case $V_1$ and $V_2$ play the roles of $A$ and $C$.  Hence, by the argument above, we have $h(V_1) = h(V_2)$. Recall that $\ell(1) = h(V_1)$.  It is easy to see that $\ell(3)=mh(V_2)$. This implies that $\ell(1) \leq \ell(3)$. 

On the other hand, if we applied this exact argument to Model II, then we get $\ell(3) \leq \ell(1)$, which implies $\ell(1) = \ell(3)$.

To show that $A$ and $C$ are isometric, it only remains to  show that they are sheared by equal amounts. Recalling again that $V_2 \subset B \cup C$ and $h(V_2) = \ell(3)$, we see that every vertical trajectory descending from $(3)$ on the top of $B$ hits $(3)$ on the bottom of $C$.  Hence, neither $A$ nor $C$ are twisted, which implies that they are isometric.  The lemma follows from Proposition \ref{P:cover}.
\end{proof}


\bibliography{mybib}{}

\providecommand{\bysame}{\leavevmode\hbox to3em{\hrulefill}\thinspace}
\providecommand{\MR}{\relax\ifhmode\unskip\space\fi MR }
\providecommand{\MRhref}[2]{%
  \href{http://www.ams.org/mathscinet-getitem?mr=#1}{#2}
}
\providecommand{\href}[2]{#2}
\begin{thebibliography}{McM06b}

\bibitem[AEM]{AEM}
Artur Avila, Alex Eskin, and Martin M\"oller, \emph{Symplectic and {I}sometric
  {$SL(2, \bR)$}-invariant subbundles of the {H}odge bundle}, preprint, arXiv
  1209.2854 (2012).

\bibitem[BM12]{BaM}
Matt Bainbridge and Martin M{\"o}ller, \emph{The {D}eligne-{M}umford
  compactification of the real multiplication locus and {T}eichm\"uller curves
  in genus 3}, Acta Math. \textbf{208} (2012), no.~1, 1--92.

\bibitem[Cal04]{Ca}
Kariane Calta, \emph{Veech surfaces and complete periodicity in genus two}, J.
  Amer. Math. Soc. \textbf{17} (2004), no.~4, 871--908.

\bibitem[EM]{EM}
Alex Eskin and Maryam Mirzakhani, \emph{Invariant and stationary measures for
  the {$SL(2,\bR)$} action on moduli space}, preprint, arXiv 1302.3320 (2013).

\bibitem[EMM]{EMM}
Alex Eskin, Maryam Mirzakhani, and Amir Mohammadi, \emph{Isolation theorems for
  {$SL(2, \bR)$}-invariant submanifolds in moduli space}, preprint,
  arXiv:1305.3015 (2013).

\bibitem[Fila]{Fi2}
Simion Filip, \emph{Semisimplicity and rigidity of the {K}ontsevich-{Z}orich
  cocycle}, preprint, arXiv:1307.7314 (2013).

\bibitem[Filb]{Fi1}
\bysame, \emph{Splitting mixed {H}odge structures over affine invariant
  manifolds}, preprint, arXiv:1311.2350 (2013).

\bibitem[HLM09]{HLM-AY}
Pascal Hubert, Erwan Lanneau, and Martin M{\"o}ller, \emph{The
  {A}rnoux-{Y}occoz {T}eichm\"uller disc}, Geom. Funct. Anal. \textbf{18}
  (2009), no.~6, 1988--2016.

\bibitem[HLM12]{HLM-Q}
\bysame, \emph{Completely periodic directions and orbit closures of many
  pseudo-{A}nosov {T}eichmueller discs in {Q}(1,1,1,1)}, Math. Ann.
  \textbf{353} (2012), no.~1, 1--35.

\bibitem[Lan08]{Lconn}
Erwan Lanneau, \emph{Connected components of the strata of the moduli spaces of
  quadratic differentials}, Ann. Sci. \'Ec. Norm. Sup\'er. (4) \textbf{41}
  (2008), no.~1, 1--56.

\bibitem[LN]{LNprym}
Erwan Lanneau and Duc-Manh Nguyen, \emph{Teichm\"uller curves generated by
  {W}eierstrass {P}rym eigenforms in genus three and genus four}, preprint,
  arXiv 1111.2299 (2011), to appear in Journal of Topology.

\bibitem[Mas82]{Ma2}
Howard Masur, \emph{Interval exchange transformations and measured foliations},
  Ann. of Math. (2) \textbf{115} (1982), no.~1, 169--200.

\bibitem[McM03]{Mc}
Curtis~T. McMullen, \emph{Billiards and {T}eichm\"uller curves on {H}ilbert
  modular surfaces}, J. Amer. Math. Soc. \textbf{16} (2003), no.~4, 857--885
  (electronic).

\bibitem[McM05]{McM:spin}
\bysame, \emph{Teichm\"uller curves in genus two: discriminant and spin}, Math.
  Ann. \textbf{333} (2005), no.~1, 87--130.

\bibitem[McM06a]{Mc2}
\bysame, \emph{Prym varieties and {T}eichm\"uller curves}, Duke Math. J.
  \textbf{133} (2006), no.~3, 569--590.

\bibitem[McM06b]{Mc4}
\bysame, \emph{Teichm\"uller curves in genus two: torsion divisors and ratios
  of sines}, Invent. Math. \textbf{165} (2006), no.~3, 651--672.

\bibitem[McM07]{Mc5}
\bysame, \emph{Dynamics of {${\rm SL}_2(\Bbb R)$} over moduli space in genus
  two}, Ann. of Math. (2) \textbf{165} (2007), no.~2, 397--456.

\bibitem[MMY]{MMY}
Carlos Matheus, Martin M\"oller, and Jean-Christophe Yoccoz, \emph{A criterion
  for the simplicity of the {L}yapunov spectrum of square-tiled surfaces},
  preprint, arXiv 1305.2033 (2013).

\bibitem[M{\"o}l]{Mprym}
Martin M{\"o}ller, \emph{Prym covers, theta functions and {K}obayashi curves in
  {H}ilbert modular surfaces}, preprint, arXiv 1111.2624 (2011), to appear in
  Amer. Journal. of Math.

\bibitem[M{\"o}l06a]{M2}
\bysame, \emph{Periodic points on {V}eech surfaces and the {M}ordell-{W}eil
  group over a {T}eichm\"uller curve}, Invent. Math. \textbf{165} (2006),
  no.~3, 633--649.

\bibitem[M{\"o}l06b]{M}
\bysame, \emph{Variations of {H}odge structures of a {T}eichm\"uller curve}, J.
  Amer. Math. Soc. \textbf{19} (2006), no.~2, 327--344 (electronic).

\bibitem[M{\"o}l08]{M3}
\bysame, \emph{Finiteness results for {T}eichm\"uller curves}, Ann. Inst.
  Fourier (Grenoble) \textbf{58} (2008), no.~1, 63--83.

\bibitem[MT02]{MT}
Howard Masur and Serge Tabachnikov, \emph{Rational billiards and flat
  structures}, Handbook of dynamical systems, {V}ol.\ 1{A}, North-Holland,
  Amsterdam, 2002, pp.~1015--1089.

\bibitem[MW]{MW}
Carlos Matheus and Alex Wright, \emph{Hodge-{T}eichm{\"u}ller planes and
  finiteness results for {T}eichm{\"u}ller curves}, preprint, arXiv 1308.0832
  (2013).

\bibitem[Ngu11]{N}
Duc-Manh Nguyen, \emph{Parallelogram decompositions and generic surfaces in {$
  H^{\rm hyp}(4)$}}, Geom. Topol. \textbf{15} (2011), no.~3, 1707--1747.

\bibitem[NW]{NW}
Duc-Manh Nguyen and Alex Wright, \emph{Non-{V}eech surfaces in
  $\mathcal{H}^{\rm hyp}(4)$ are generic}, preprint, arXiv:1306.4922 (2013).

\bibitem[SW04]{SW2}
John Smillie and Barak Weiss, \emph{Minimal sets for flows on moduli space},
  Israel J. Math. \textbf{142} (2004), 249--260.

\bibitem[Vee82]{V2}
William~A. Veech, \emph{Gauss measures for transformations on the space of
  interval exchange maps}, Ann. of Math. (2) \textbf{115} (1982), no.~1,
  201--242.

\bibitem[Vee89]{V}
W.~A. Veech, \emph{Teichm\"uller curves in moduli space, {E}isenstein series
  and an application to triangular billiards}, Invent. Math. \textbf{97}
  (1989), no.~3, 553--583.

\bibitem[War98]{W}
Clayton~C. Ward, \emph{Calculation of {F}uchsian groups associated to billiards
  in a rational triangle}, Ergodic Theory Dynam. Systems \textbf{18} (1998),
  no.~4, 1019--1042.

\bibitem[Wria]{Wcyl}
Alex Wright, \emph{Cylinder deformations in orbit closures of translation
  surfaces}, preprint, arXiv 1302.4108 (2013).

\bibitem[Wrib]{Wfield}
\bysame, \emph{The field of definition of affine invariant submanifolds of the
  moduli space of abelian differentials}, preprint, arXiv 1210.4806 (2012), to
  appear in Geom. Top.

\bibitem[Wri13]{W2}
\bysame, \emph{Schwarz triangle mappings and {T}eichm\"uller curves: the
  {V}eech-{W}ard-{B}ouw-{M}\"oller curves}, Geom. Funct. Anal. \textbf{23}
  (2013), no.~2, 776--809.

\bibitem[Zor06]{Z}
Anton Zorich, \emph{Flat surfaces}, Frontiers in number theory, physics, and
  geometry. {I}, Springer, Berlin, 2006, pp.~437--583.

\end{thebibliography}
\bibliographystyle{amsalpha}
\end{document}